\documentclass[a4paper,12pt, reqno]{amsart}

\usepackage{amsmath, amsthm, amscd, amsfonts, amssymb, graphicx, color}
\usepackage[bookmarksnumbered, colorlinks, plainpages]{hyperref}

\usepackage{enumitem}
\usepackage{url}
\usepackage{multirow}
\usepackage{graphicx}
\usepackage{subfigure}
\usepackage[pass]{geometry}
\usepackage{cite}
\usepackage{cancel}
\usepackage{color}
\usepackage{xcolor}
\usepackage{pgf,tikz}
\usetikzlibrary{arrows}
\usepackage{lscape}%{pdflscape}
\tikzstyle{v} = [circle, draw, inner sep=2pt, minimum size=3pt, fill=black]
\usetikzlibrary{calc,shapes,decorations.pathreplacing,matrix}
\usetikzlibrary{patterns}

\tikzset{square matrix/.style={
    matrix of nodes,
    column sep=-\pgflinewidth, row sep=-\pgflinewidth,
    nodes={draw,
      minimum height=4.5pt,
      anchor=center,
      text width=4.5pt,
      align=center,
      inner sep=0pt
    },
  },
  square matrix/.default=1.2cm
}

\newtheorem{Theorem}{Theorem}[section]
\newtheorem{Definition}[Theorem]{Definition}

\newtheorem{Lemma}[Theorem]{Lemma}

\newtheorem{Corollary}[Theorem]{Corollary}
\newtheorem{Remark}[Theorem]{Remark}
\newtheorem{Example}[Theorem]{Example}
\newtheorem{Problem}[Theorem]{Problem}

\DeclareMathOperator{\leaf}{leaf}
\DeclareMathOperator{\diam}{diam}

\begin{document}

\title{Total Dominator coloring  number of middle graphs}

\author[F. Kazemnejad]{Farshad Kazemnejad}
\address{Farshad Kazemnejad, Department of Mathematics, Faculty of Basic Sciences, Ilam University, P.O. Box 69315-516, Ilam, Iran.}
\email{kazemnejad.farshad@gmail.com}
\author[B. Pahlavsay]{Behnaz Pahlavsay}
\address{Behnaz Pahlavsay, Department of Mathematics, Hokkaido University, Kita 10, Nishi 8, Kita-Ku, Sapporo 060-0810, Japan.}
\email{pahlavsay@math.sci.hokudai.ac.jp}
\author[E. Palezzato]{Elisa Palezzato}
\address{Elisa Palezzato, Department of Mathematics, Hokkaido University, Kita 10, Nishi 8, Kita-Ku, Sapporo 060-0810, Japan.}
\email{palezzato@math.sci.hokudai.ac.jp}
\author[M. Torielli]{Michele Torielli}
\address{Michele Torielli, Department of Mathematics, GI-CoRE GSB, Hokkaido University, Kita 10, Nishi 8, Kita-Ku, Sapporo 060-0810, Japan.}
\email{torielli@math.sci.hokudai.ac.jp}

\date{\today}

\maketitle

\begin{abstract}
A total dominator coloring of a graph $G$ is a proper coloring of
$G$ in which each vertex of the graph is adjacent to every vertex of
some color class. The total dominator chromatic number of a graph is the minimum number of color classes in a total dominator
coloring. In this article, we study the total dominator coloring on middle graphs by giving several bounds for the case of general graphs and  trees. Moreover, we  calculate explicitely the total dominator chromatic number of the middle graph of several known families of graphs.
%
%by giving some tight bounds for the total dominator chromatic number of the middle of a graph, join of a graph with an empty graph and Nordhaus-Gaddum-like relations. Also we will calculate  the total dominator chromatic number of the middle of a path, a cycle, a wheel, a complete graph and a complete multipartite graph.
\\[0.2em]

\\[0.2em]

\end{abstract}

%------------------------------------------------------------------------------------%
%\pagestyle{myheadings}
%\markboth{\centerline {\scriptsize  F. Kazemnejad, B. Pahlavsay, E. Palezzato and  M. Torielli}}     {\centerline {\scriptsize  F. Kazemnejad, B. Pahlavsay, E. Palezzato and  M. Torielli}}

%------------------------------------------------------------------------------------

\section{Introduction}

All graphs considered in this paper are non-empty, finite, undirected and simple. For standard graph theory terminology not given here we refer to \cite{West}. 
For a simple graph $G=(V(G),E(G)) $ we will denote the
\emph{open neighbourhood} and the \emph{closed neighbourhood} of a
vertex $v\in V(G)$ by $N_{G}(v)=\{u\in V(G)\ |\ uv\in E(G)\}$ and
$N_{G}[v]=N_{G}(v)\cup \{v\}$, respectively, the
\emph{minimum} and \emph{maximum degree} of $G$ by
$\delta =\delta (G)$ and $\Delta =\Delta (G)$, respectively, and the \emph{induced subgraph} by $S\subset V(G)$ by $G[S]$.
%Given $S\subseteq V(G)$, we write $G[S]$ to denote the subgraph of $G$ \emph{induced} by $S$.
%We write $K_{n}$, $C_{n}$ and $P_{n}$ for a \emph{complete graph}, a \emph{cycle} and a \emph{path} of order $n$, respectively, while $G[S]$, $W_n$ and $K_{n_1,n_2,\cdots,n_p}$ denote the subgraph of $G$ \emph{induced} by a vertex set $S$, a \emph{wheel} of order $n+1$, and a \emph{complete $p$-partite graph}, respectively. 

%\vskip 0.2 true cm

%-------------------------------------------------------------
%\vskip 0.2 true cm

%\textbf{Total domination number.} 
The notion of domination is well studied in graph theory and the literature
on this subject has been surveyed in the two books \cite{hhs1, hhs2}, see also \cite{PPTdominlatin} and \cite{KPPTdomin}. 
A famous generalization of domination is the notion of total domination, see \cite{HeYe13}, \cite{PPT3dominrook} and \cite{KPPT}. A \emph{total dominating set} $S$ of a graph $G$ is a subset of $V(G)$ such that for each vertex $v$ we have $N_G(v)\cap
S\neq \emptyset$. The \emph{total domination number $\gamma_t(G)$} of $G$ is the minimum cardinality of a total dominating set of $G$. 

%----------------------------------------------------------------
%\vskip 0.2 true cm
Similarly to the notion of domination, also the notion of coloring of graphs has been intensively studied.
%\textbf{Total dominator Coloring.} 
\begin{Definition} A \emph{proper coloring} of a graph $G$ is a function from $V(G)$ to a set of colors such that any two
adjacent vertices have a different color. The \emph{chromatic number}
$\chi (G)$ of $G$ is the minimum number of colors needed in a proper
coloring of $G$.
\end{Definition} 
In a proper coloring of a graph, a \emph{color class} is a set consisting of all those vertices assigned the same
color. If $f$ is a proper coloring of $G$ with the coloring classes
$V_1, \dots, V_\ell$ such that every vertex in $V_i$ has
color $i$, we write simply $f=(V_1,\dots,V_{\ell})$. 

Similarly to the notion of chromatic number, we can recall the notion of edge chromatic number.
\begin{Definition} 
The \emph{edge chromatic number}, sometimes also called the \emph{chromatic index}, of a graph $G$ is the smallest number of
 colors necessary to color each edge of $G$ such that no two edges incident on the same vertex have 
the same color. In other words, it is the number of distinct colors in a minimum edge coloring and it is denoted by $\chi'(G)$.
\end{Definition}

%\vskip 0.2 true cm

Graph coloring is used as a model for a vast
number of practical problems involving allocation of scarce
resources and it has played a key role in
the development of graph theory and, more generally, discrete
mathematics and combinatorial optimization. %Graph colorability is NP-complete in the general case, although the problem is solvable in polynomial time for many classes \cite{GJ}. 

%\vskip 0.2 true cm

Motivated by the relation between coloring and domination, the notion of total dominator colorings was introduced in \cite{Kaz2015}. For more information see \cite{1Kaz,2Kaz,3Kaz,4Kaz,Hen2015}.

%---------------------------------------------------------------------------------------------------------------
\begin{Definition}[\cite{Kaz2015}]
\label{total dominator coloring}  A \emph{Total Dominator Coloring},
briefly TDC, of a graph $G$ with a positive minimum degree is a proper coloring of $G$ in
which each vertex of $G$ is adjacent to every vertex of some
color class. The \emph{total dominator chromatic number} $\chi_d^t(G)$
of $G$ is the minimum number of color classes in a TDC of $G$.
\end{Definition}

\begin{Example} Consider $G=P_4$ with $V(G)=\{v_1,\dots, v_4\}$ and $E(G)=\{v_1v_2,v_2v_3,v_3v_4\}$. If we consider $V_1=\{v_1,v_3\}$ and $V_2=\{v_2,v_4\}$, then $f=(V_1,V_2)$ is a proper coloring of $G$ but it is not a total dominator coloring, since $v_1$ is not adjacent to every vertex of some color class. On the other hand, if we consider $V_1=\{v_1,v_4\}$, $V_2=\{v_2\}$ and $V_3=\{v_3\}$, then $f=(V_1, V_2, V_3)$ is a total dominator coloring of $G$.
\end{Example}

%---------------------------------------------------------------------------------------------------------------
\begin{Definition} Let $f=(V_1,\dots,V_{\ell})$ be a total dominator coloring  of a graph $G$. If a vertex $v\in V(G)$ satisfies $V_i \subseteq N_G(v)$, then $v$ is called a \emph{common neighbour} of
$V_i$ or we say that $V_i$ \emph{totally dominates} $v$. In this case, we write $v\succ_t V_i$, otherwise, we write $v \not\succ_t  V_i$. 
\end{Definition}

The set of all common neighbours of $V_i$ with respect to $f$ is called the \emph{common neighbourhood} of $V_i$ in $G$ and denoted by $CN_G(V_i)$ or simply by $CN(V_i)$. 
A vertex $v$ is called a \emph{private neighbour} of $V_i$ with respect to $f$ if $v\succ_t V_i$ and $v\nsucc_t V_j$ for all $j\neq i$.
Moreover, for any total dominator coloring $f=(V_1,\dots,V_{\ell})$ of a graph $G$, we have 
\begin{eqnarray}\label{cup CN(V_i)=V(G)}
\bigcup\limits_{i=1}^{\ell} CN(V_i)=V(G).
\end{eqnarray}

In \cite{HamYos}, the authors introduced the notion of the middle graph $M(G)$ of a graph $G$ as an intersection graph on $V(G)$.

%-------------------------------------------------------------
\begin{Definition}
Let $G=(V(G),E(G)) $ be a simple graph.
 The \emph{middle graph} $M(G)$ of a graph $G$ is the graph whose vertex set is $V(G)\cup E(G)$ and two vertices $x, y$ in the vertex set of $M(G)$ are adjacent in $M(G)$ in case one the following holds
 \begin{enumerate} 
 \item $x, y$ are in $E(G)$ and $x, y$ are adjacent in $G$. 
 \item $x$ is in $V(G)$, $y$ is in $E(G)$, and $x, y$ are incident in $G$. 
 \end{enumerate}
 \end{Definition}
 
It is obvious that $M(G)$ contains the line graph $L(G)$ as induced subgraph, and that if $G$ is a graph of order $n=|V(G)|$ and size $m=|E(G)|$, then $M(G)$ is a graph of order $n+m$ and size $2m+|E(L(G))| $ which is obtained by subdividing each edge of $G$ exactly once and joining all the adjacent edges of $G$ in $ M(G)$. 

In order to avoid confusion throughout the paper, we fix a ``standard'' notation for the vertex set and the edge set of $M(G)$. Assume $V(G)=\{v_1,\dots, v_n\}$, then we set $V(M(G))=V(G)\cup \mathcal{M}$, where $\mathcal{M}=\{m_{ij}~|~ v_iv_j\in E(G)\}$ and $E(M(G))=\{v_im_{ij},v_jm_{ij}~|~ v_iv_j\in E(G)\}\cup E(L(G)) $.

%---------------------------------------------------------------

%\vspace{0.2cm}

The goal of this paper is to study the total dominator chromatic number of middle graphs.
%In more details, while we give some tight bounds for the total dominator chromatic number of the middle of a connected or disconnected graph or join of a graph with an empty graph 
In Section 2, we describe some useful bound involving total dominator chromatic number. In Section 3, we prove a series of lemmas that will play an important role in the rest of the paper. In Section 4, we calculate the total dominator chromatic number of the middle graph of several known families of graphs. In Section 5, we describe bounds for the total dominator coloring number of the middle graph of trees.
%we give the exact amount of  the total dominator chromatic number of the middle of a path, a cycle, a wheel and a complete multipartite graph in Section 3. Finally we give some Nordhaus-Gaddum relations on the total dominator chromatic numbers of the middle of a graph with its complement.

%\vspace{0.2cm}

%---------------------------------------------------------------------------------------------------------------

%------------------------------------------- 
%-------------------------------------------------------------------------------
%\begin{prop}\emph{\cite{Kaz2015}}\label{chi_d^t =<g_t+min{chi(G-S)}} For any connected graph $G$ of order $n$ with $\delta(G)\geq 1$,
%\[ \chi_d^t(G) \leq \gamma_t(G)+ \min_S \chi(G[V(G)-S]),\] in which $S$ is a min-TDS of $G$, and the bound is sharp.
%\end{prop}

%----------------------------------------------------------------------------------------------------------------

%----------------------------------------------------------
\section{Useful bounds }
We start this section by recalling some known bounds for total dominator coloring numbers.

\begin{Theorem}[\cite{Kaz2015}]
\label{2=<chi_d ^t=<n} For any connected graph $G$ of order $n$ with $\delta(G)\geq 1$, 
\[
\max\{\chi(G),\gamma_t(G),2\}\leq \chi_d ^t(G)\leq n.
\]
Furthermore, $\chi_d ^t(G)=2$ if and only if $G$ is a
complete bipartite graph, and $\chi_d ^t(G)=n$ if and only if $G$ is a complete graph.
\end{Theorem}

\begin{Theorem}[\cite{Kaz2015}]
\label{chi_d ^t=<gamma+min chi} 
For any connected graph $G$ of order $n$ with $\delta(G)\geq 1$, 
$$\chi_d ^t(G)\leq \gamma_t(G)+\min\{\chi(G[V(G)-S])~|~S \text{ is a min-TDS of } G\},$$
%where $S\subseteq V(G)$ is a min-TDS of $G$. 
Moreover, $\chi_d ^t(G)\leq \gamma_t(G)+\chi(G)$.
\end{Theorem}

%\begin{Remark}
%Notice that a minimal total dominator coloring of $G$ is not unique.
In Theorem~\ref{chi_d ^t=<gamma+min chi} we have to consider a minimum in order to get a better bound. In fact, if we consider $G=P_6$ with vertex set $\{v_1,\dots,v_6\}$, then $S_1=\{v_2,v_3,v_4,v_5\}$ and $S_2=\{v_1,v_2,v_5,v_6\}$ are two minimal TDS sets of $P_6$, with $\chi(G[V(G)-S_1])=1$ and $\chi(G[V(G)-S_2])=2$.
%\end{Remark}

%\begin{Theorem}[\cite{CZ}(K$\ddot{o}$nig's Theorem)]
% For any nonempty bipartite graph $G$,
% $$\chi'(G)=\Delta(G).$$
%\end{Theorem}

%\begin{Theorem}[\cite{Harary}]
% For any nonempty graph $G$,
%$$\Delta(G)\leq\chi'(G)\leq \Delta(G)+1$$
%\end{Theorem}
%-------------------------------------------------
%\begin{Theorem}[\cite{Clark}]\label{chi'  (K_n)=n,n-1}
%	For any complete graph  $K_{n}$ of order at least 2,
%	\begin{equation*}
%	\chi'  (K_n)=\left\{
%	\begin{array}{ll}
%	n-1 & \mbox{if $n$ is even}\\
%	n  & \mbox{if $n$ is odd}.
%	\end{array}
%	\right.
%	\end{equation*}
%\end{Theorem}

%------------------------------------------
%----------------------------------------
By Theorem \ref{2=<chi_d ^t=<n}, we have the following result.
\begin{Theorem} 
\label{2=<chiM_d ^t=<n} For any connected graph $G$ of order $n\geq 2$ and size $m$, 
\[
\max\{\chi(M(G)),\gamma_t(M(G))\}\leq \chi_d ^t(M(G))\leq n+m-1.
\]
\end{Theorem}
From \cite{Kaz2015}, it follows that for any graph $G$, with connected components $G_1, \dots,G_w$, which has no isolated vertex, we have
\[
\max\limits_{i=1}^{w}\{\chi_d^t(G_i)\}+2w-2\leq \chi_d ^t(G) \leq \sum_{i=1}^{w}\chi_d ^t(G_i).
\]
Moreover, since $M(G)=M(G_1)+\cdots+M(G_w)$, we obtain the following result.

\begin{Theorem}  For any graph $G$ with connected components $G_1, \dots,G_w$ which has no isolated vertex, we have
\[
\max\limits_{i=1}^{w}\{\chi_d^t(M(G_i))\}+2w-2\leq \chi_d ^t(M(G)) \leq \sum_{i=1}^{w}\chi_d ^t(M(G_i)).
\]
\end{Theorem}
Therefore, it is sufficient to verify the total dominator chromatic number of connected graph.
The next theorem describes  bounds for the total domination number of the middle graph.
%---------------------------------------------------------

\begin{Theorem}[\cite{KPPT}]
\label{2n/3 =<  gamma =< n-1}
	Let $G$ be a connected graph with $n \geq 3$ vertices. Then
	$$\lceil \frac{2n}{3}\rceil\leq\gamma_t(M(G))\leq n-1.$$
\end{Theorem}

%--------------------------------------------------
As an immediate consequence of Theorems \ref{2=<chiM_d ^t=<n} and \ref{2n/3 =<  gamma =< n-1}, we have the following result.
%------------------------------------------------------------
%\vspace{2cm}
\begin{Corollary}\label{chiM_d ^t=>2n/3}
Let $G$ be a connected graph with $n \geq 3$ vertices. Then
$$\chi_d ^t(M(G)) \geq \lceil \frac{2n}{3}\rceil.$$
\end{Corollary}

\begin{Problem}
Classify all connected graphs $G$ such that $\chi_d ^t(M(G)) = \lceil \frac{2n}{3}\rceil.$
\end{Problem}
%------------------------------------------------
%------------------------------------------
%\begin{Problem}
%\textcolor{red}{Characterize graphs $G$ satisfies $\chi_d ^t(M(G)) =\lceil \frac{2n}{3}\rceil.$}
%\end{Problem}
%%---------------------------------------------------------
%\begin{Problem}
%	\textcolor{blue}{whether the lower bound in Theorem \ref{2=<chiM_d ^t=<n} is tight?\\
%		Characterize graphs $G$ satisfies $\chi_d ^t(M(G)) =\chi(M(G))$ or,$\chi_d ^t(M(G)) =\gamma_t(M(G))$?}
%\end{Problem}

%---------------------------------------
%-----------------------------------------------------

\section{First general results}

We start the section by recalling the notion of an independent set of vertices that is closely related to the notion of coloring.
\begin{Definition}
An \emph{independent set} is a set of vertices in $G$ that no two of which are adjacent. %That is, $S$ is a set of vertices such that for every two vertices in $S$ there is no edge connecting the two. 
A \emph{maximum independent set} is an independent set of largest possible size for a given graph $G$. This size is called the \emph{independence number} of $G$ and it is denoted by $\alpha(G)$.
\end{Definition}
\begin{Lemma}\label{a(M(G))=n}
For any connected graph $G$ of order $n\geq 2$, we have
$$\alpha(M(G))=n.$$
\end{Lemma}
\begin{proof}
Let $S$ be an independent set of $M(G)$. Since $V=\{v_1,\dots,v_n\}$ is  an independent set of $M(G)$, then $\alpha(M(G))\geq n$. On the other hand, consider $M=\{m_{ij}|v_iv_j\in E(G)\}$ and $t=|M\cap S|$. Since if $m_{ij}\in S$, then $v_i,v_j\notin S$, we have $|V\cap S|\leq n-t$, and hence that $|S|=|V\cap S|+|M\cap S|\leq n-t+t= n$. As a consequence, $\alpha(M(G))=n$.
\end{proof}

We can now state several lemmas that will play an important role in the rest of the paper. Assume that $G$ is a simple graph, $V(M(G))=V\cup M$, where $V=V(G)=\{v_1,\dots,v_n\}$, and $M=\{m_{ij}|v_iv_j\in E(G)\}$.

\begin{Lemma}\label{3item}
Let  $f=(V_1, \dots,V_{\ell})$ be a TDC of $M(G)$. Then 
\begin{enumerate}
  \item For any vertex $v_i\in V$, if $v_i\succ_t V_k$ for some $1\leq k\leq \ell$, then $|V_k|=1$ and $V_k=\{m_{ij}\}$ for some $j$.
  \item  For any $m_{ij}\in M$, if $m_{ij}\succ_t V_k$ for some $1\leq k\leq \ell$, then $1\leq|V_k|\leq2$.
  \item For any $1\leq k\leq \ell$, we have $|V_k|\leq \alpha (M(G))=n$.
%\item Let $A_i=\{V_k~|~ |V_k|=i\}$. Then  $|A_1|\geq \lceil \frac{n}{2}\rceil$.(notice if $v_i\succ_t V_k$ for some $1\leq k\leq \ell$, then $V_{k} \subset N(v_i) \subset M$ and $|V_k|=1$.  ) 
\end{enumerate}
\end{Lemma}
\begin{proof} 1) If $v_i\succ_t V_k$, then $V_k \subseteq N_{M(G)}(v_i) \subseteq M$. This implies that if $|V_k|\ge 2$, then there exist $j\ne r$ such that $m_{ij}, m_{ir}\in V_k$, but this is impossible by the definition of TDC.

2) If $m_{ij}\succ_t V_k$, then $V_k \subseteq N_{M(G)}(m_{ij})$. This implies that if $|V_k|\ge 3$, then there exist $s\ne r$ such that $m_{is}, m_{ir}\in V_k$ or there exists $p\ne j$ such that $v_i, m_{ip}\in V_k$. However, both cases are impossible by the definition of TDC.

3) This is a consequence of the definition of independent set and Lemma~\ref{a(M(G))=n}.
\end{proof}

%---------------------------------------------------
%-------------------------------------------------
\begin{Lemma}\label{4item}
	Let  $f=(V_1,\dots,V_{\ell})$ be a TDC of $M(G)$. Then  
	\begin{enumerate}
		\item if $|V_i| \geq 3$, then $|CN(V_i)|=0$.% where $|V_i| \geq 3$ for every $i$. %(Because if $v\succ_t V_k$, then $|V_k|\leq2$ for $v \in V(M(G))$ )
		\item  $\bigcup\limits_{\substack{i=1\\|V_i|\le 2}}^{\ell} CN(V_i)=V(M(G))$. %for some $1\leq k\leq \ell$ where $|V_i| \leq 2$ for every $1\leq i\leq k$.  %(by relation \eqref{cup CN(V_i)=V(G)})
		%\item If $|V_k|=2$, then $|CN(V_k)|=1$. 
		%\item If $V_k=\{v_i\}$ for some $i$, then $|CN(V_k)|=2$.
		%\item If $V_k=\{m_{ij} \}$ for some $i,j$, then $|CN(V_k)|=4$.
		%\item  $|A_1|+|A_2| \le \ell$.
	\end{enumerate}
\end{Lemma}
\begin{proof} 1) It is a consequence of the fact that if $v \in V(M(G))$ and $v\succ_t V_k$, then $|V_k|\leq2$, by Lemma~\ref{3item}.

2) It is a direct consequence of relation \eqref{cup CN(V_i)=V(G)} and the previous part of the lemma.
\end{proof}

Given $f=(V_1,\dots,V_{\ell})$ a TDC of $M(G)$, we will denote by $A_i=\{V_k~|~ |V_k|=i\}$.

\begin{Lemma}\label{lemma:A_i} Let  $f=(V_1,\dots,V_{\ell})$ be a TDC of $M(G)$. Then
\begin{enumerate}
\item $|A_1|\geq \lceil \frac{n}{2}\rceil$.
\item $|A_1|+|A_2| \le \ell$.
\end{enumerate}
\end{Lemma}
\begin{proof}
1) Since $f$ is a TDC, then for every $i=1,\dots,n$, there exists $1\le k\le \ell$ such that $v_{i}\succ_t V_k$. By Lemma~\ref{3item}, $V_k=\{m_{ij}\}$ for some $j$. Now since $v_{t}\nsucc_t \{m_{ij}\}$ for $t\neq i,j$, then for every two vertices $v_{i}, v_{j}$ there exists at least one $k$ such that $|V_k|=1$, and hence $|A_1|\geq \lceil \frac{n}{2}\rceil$.
	 
2) It is a direct consequence of $|A_1|+|A_2|\le \sum_{i\ge1}|A_i|= \ell$.
\end{proof}
%-------------------------------------------------------------------
Since the middle graph $M(G)$ contains the line graph $L(G)$, we can relate their total dominator chromatic numbers.
\begin{Theorem} \label{chi_d^t(M(G))geq chi_d ^t(L(G))}
	  Let $G$ be connected graph of order $n\geq 2$, size $m\ge2$ and $\delta(G)\geq 1$. Then
$$\chi_d^t(M(G))\geq \chi_d ^t(L(G)).$$
%Moreover, this bound is tight.
\end{Theorem}

\begin{proof}
To fix the notation, assume that $V(G)=V=\{v_1,\dots,v_n\}$. Then $V(M(G))=V\cup M$, where $M=V(L(G))=\{m_{ij}~|~v_iv_j\in E(G)\}$. Let $f=(V_1,\dots,V_{\ell})$ be a TDC of $M(G)$.
For all $1\le i\le \ell$, consider $W_i=V_i\setminus V$ and $g=(W_1,\dots,W_\ell)$. Then $g$ define a coloring of the line graph $L(G)$.
If for every $m_{ij}\in M$, there exists $i\le k\le\ell$ such that $m_{ij}\succ_t V_k$ and $V_k\cap M\ne\emptyset$, then $g$ is clearly a TDC of $L(G)$. Assume now there exists $m_{ij}\in M$ such that  $m_{ij}\succ_t V_{k_1},\dots, V_{k_r}$, but $V_{k_1}\cup\cdots\cup V_{k_r}\subseteq V$. Since $G$ is connected and $m\ge2$, this implies $N_{M(G)}(m_{ij})\cap M=\{m_1,\dots, m_p\}\ne \emptyset$. For every $1\le j\le p$, assume that $m_j\in V_{t_j}$. If for some $1\le j\le p$, we have $V_{t_j}\cap M\subseteq N_{M(G)}(m_{ij})$, then there is nothing to do (in this case $m_{ij}\succ_t W_{t_j}$). Assume now that for all $1\le j\le p$, we have $V_{t_j}\cap M\nsubseteq N_{M(G)}(m_{ij})$. This implies that $|V_{t_j}\cap M|\ge 2$. In particular, $V_{t_1}\cap M\nsubseteq N_{M(G)}(m_{ij})$ and $|V_{t_1}\cap M|\ge 2$. Redefine $W_{k_1}=\{m_1\}$ and $W_{t_1}=V_{t_1}\setminus(V\cup\{m_1\})$ (in this way $m_{ij}\succ_t W_{k_1}$). If we apply this procedure to all $m_{ij}\in M$ such that  $m_{ij}\succ_t V_{k_1},\dots, V_{k_r}$ with $V_{k_1}\cup\cdots\cup V_{k_r}\subseteq V$, we have that the obtained $g$ is a TDC of $L(G)$. This implies that $\chi_d^t(M(G))\geq \chi_d ^t(L(G))$.
\end{proof}

\begin{Problem}
Classify all connected graphs $G$ such that $\chi_d^t(M(G))= \chi_d ^t(L(G))$.
\end{Problem}
%----------------------------------------------------
%\begin{Question}
%\textcolor{red}{Show that the lower bound in Theorem \ref{chi_d^t(M(G))geq chi_d ^t(L(G))} is sharp?}
%\end{Question}
%----------------------------------------------------------------
\section{Middle graph of known families}\label{t80}
In this section, we calculate the total dominator chromatic number of the middle graph of several known families of graphs. %The total dominator chromatic number of the middle of cycles and paths are given in the first two propositions. 
\begin{Theorem}\label{theo:TDCstargraph}  For any star graph $K_{1,n}$ on $n+1$ vertices, with $n\geq 3$, 
\[
\chi_d^t(M(K_{1,n}))=n+1.
\]
\end{Theorem}
\begin{proof}
To fix the notation, assume $V(K_{1,n})=\{v_0,v_1,\dots,v_n\}$ and $E(K_{1,n})=\{v_0v_1,\dots,v_0v_n\}$. Then $V(M(K_{1,n}))=V(K_{1,n})\cup M$ where $M=\{m_i~|~1\leq i\leq n\}$. Let $f=(V_1,\dots,V_{\ell})$ be a minimal TDC of $M(K_{1,n})$. Since $K_{n+1}\cong M(K_{1,n})[v_0,m_1,\dots, m_n] \subseteq M(K_{1,n})$, we have $\ell\geq n+1$. Consider $V_i=\{m_i\}$ for $1\leq i\leq n$, $V_{n+1}=V(K_{1,n})$ and $g=(V_1,\dots,V_{n+1})$. By construction, $g$ is a TDC of $ M(K_{1,n})$, and hence, $\chi_d^t(M(K_{1,n}))=n+1$.
\end{proof}

\begin{Theorem} \label{chi_d^t(M(S_{1,n,n}))=2n+1}
	  For any double star graph $S_{1,n,n}$ on $2n+1$ vertices, with $n\geq 1$, 
\[
\chi_d^t(M(S_{1,n,n}))=2n+1.
\]
\end{Theorem}
\begin{proof}
To fix the notation, assume that $V(S_{1,n,n})=\{v_0,v_1,\dots,v_{2n}\}$ and $E(S_{1,n,n})=\{v_0v_i,v_iv_{n+i}~|~1\le i\le n\}$. 
Then $V(M(S_{1,n,n}))=V(S_{1,n,n})\cup M$, where $M=\{m_{i}, m_{i(n+i)}~|~1 \leq i\leq n\}$. 
Let $f=(V_1, \dots, V_{\ell})$ be a TDC of $M(S_{1,n,n})$. Since the subgraph of $ M(S_{1,n,n})$ induced by $\{m_i~|~ 1\leq i\leq n\}\cup \{v_0\}$ is isomorphic to a complete graph of order $n+1$, then $\chi_d^t(M(S_{1,n,n}))\geq n+1$. Without loss of generality, we can assume that $m_{i} \in V_i$, for $1 \leq i \leq n$, and $v_{0} \in V_{n+1}$.

Now since $N_{M(S_{1,n,n})}(v_{n+i})=\{m_{i(n+i)}\}$, then each $m_{i(n+i)}$ belong to a color class composed of only one element.
%and $m_{i(n+i)}\not\succ_t  (V_1\cup V_2\cup\cdots \cup V_{n+1})$ and also $\{m_{i(n+i)}\}$ and $\{m_{i}\}$ are adjacent in $S_{1,n,n}$ for $1 \leq i\leq n$, 
This implies that we need at least another $n$ colors for a total dominator coloring of $S_{1,n,n}$. Hence $\chi_d^t(M(S_{1,n,n}))\geq 2n+1$.

 On the other hand, since  $g=(V_1,\dots,V_{2n+1})$, where $V_i=\{m_i\}$,  $V_{n+i}=\{m_{i(n+i)}\}$  for $1\leq i\leq n$ and $V_{2n+1}=V(S_{1,n,n})$, is a TDC of $M(S_{1,n,n})$, then $\chi_d^t(M(S_{1,n,n}))\leq 2n+1$ and hence $\chi_d^t(M(S_{1,n,n}))=2n+1$.
\end{proof}

\begin{Theorem} \label{chi_d^t(M(P))} For any path $P_n$ of order $n\geq 3$, 
$$
\chi_d^t(M(P_n))=
\begin{cases}
n& \text{ if $3\leq n\leq 7$}\\ 
n-1& \text{ if $n=8$}\\ 
\lceil \frac{2n}{3}\rceil+2 & \text{ otherwise} 
\end{cases}
$$
\end{Theorem}

\begin{proof}
%Let $P_n$ be a path of order $n\geq 3$, with the vertex set
Assume $V=V(P_n)=\{v_1,\dots,v_n\}$. Then $V(M(P_n))=V\cup M$, where
$M=\{m_{i(i+1)}|1\leq i\leq n-1\}$. Notice that if $f=(V_1,\dots, V_\ell)$ is a TDC of $M(P_n)$, then we can always assume that $V_1=\{m_{12}\}$ and $V_2=\{m_{(n-1)n}\}$. Moreover, since defining $V_i=\{m_{(i-1)i}\}$, for $i=3,\dots,n-1$, and $V_n=V$ always gives a TDC of $M(P_n)$, then $\chi_d^t(M(P_n))\le n$

Assume first that $n=3$. Since the induced subgraph $M(P_3)[m_{12},v_2,m_{23}]$ is isomorphic to $K_3$, then $\chi_d^t(M(P_3))\ge3$, %and $(\{m_{12}\},\{m_{23}\},\{v_1,v_2,v_3\})$ is a TDC of $M(P_3)$, then 
and hence$\chi_d^t(M(P_3))=3$. 

Assume that $n=4$. By Corollary~\ref{chiM_d ^t=>2n/3}, $\chi_d^t(M(P_4))\ge3$. If $\chi_d^t(M(P_4))=3$, this would force $V_3=V\cup\{m_{23}\}$, but this is impossible. This implies that $\chi_d^t(M(P_4))=4$. %On the other hand, $(\{m_{12}\},\{m_{23}\},\{m_{34}\}\,{v_1,v_2,v_3,v_4\})$ is a TDC of $M(P_4)$, and hence $\chi_d^t(M(P_4))=4$. 

Fix $n=5$. By Corollary~\ref{chiM_d ^t=>2n/3}, $\chi_d^t(M(P_5))\ge4$. If $\chi_d^t(M(P_5))=4$, then we can assume $m_{23}\in V_3$ and $m_{34}\in V_4$, but then we could not color $v_3$. This implies that $\chi_d^t(M(P_5))=5$.  

When $n=6$, by Corollary~\ref{chiM_d ^t=>2n/3}, $\chi_d^t(M(P_6))\ge4$. If $\chi_d^t(M(P_6))=4$, then we can assume $m_{23}\in V_3$ and $m_{34}\in V_4$, but then we could not color $v_3$. This implies that $\chi_d^t(M(P_6))\ge5$. If $\chi_d^t(M(P_6))=5$, then we can assume $m_{23}\in V_3$, $m_{34}\in V_4$ and $v_3\in V_5$. However, this implies that $m_{56}\nsucc_t V_i$ for all $i=1,3,4,5$, but this is impossible, and hence, $\chi_d^t(M(P_6))=6$.

Consider $n=7$. By Corollary~\ref{chiM_d ^t=>2n/3}, $\chi_d^t(M(P_7))\ge5$. If $\chi_d^t(M(P_7))=5$, then we can assume $m_{23}\in V_3$, $m_{34}\in V_4$ and $v_3\in V_5$. This implies that $m_{67}\nsucc_t V_i$ for all $i=1,3,4,5$, but this is impossible, and hence, $\chi_d^t(M(P_7))\ge 6$. If $\chi_d^t(M(P_7))=6$, then we can assume $m_{23}\in V_3$, $m_{34}\in V_4$ and $v_3\in V_5$, and hence that $m_{67}\succ_t V_6$. This forces $m_{45}\in V_3\cup V_5$. If $m_{45}\in V_3$, then $m_{12}\nsucc_t V_i$ for all $i=2,\dots, 6$, and hence $m_{45}\in V_5$. As a consequence, $v_4\in V_3$, but this implies $m_{12}\nsucc_t V_i$ for all $i=2,\dots, 6$. This shows that $\chi_d^t(M(P_7))=7$.

Assume that $n=8$. By Corollary~\ref{chiM_d ^t=>2n/3}, $\chi_d^t(M(P_8))\ge6$. If $\chi_d^t(M(P_8))=6$, then we can assume $m_{23}\in V_3$, $m_{34}\in V_4$ and $v_3\in V_5$, and hence that $m_{78}\succ_t V_6$. This forces $m_{45}\in V_3\cup V_5$. If $m_{45}\in V_3$, then $m_{12}\nsucc_t V_i$ for all $i=2,\dots, 6$, and hence $m_{45}\in V_5$. As a consequence, $v_4\in V_3$, but this implies $m_{12}\nsucc_t V_i$ for all $i=2,\dots, 6$. This shows that $\chi_d^t(M(P_8))\ge7$. On the other hand, if we consider $V_1=\{m_{12}\}$, $V_2=\{m_{78}\}$, $V_3=\{m_{23}\}$, $V_4=\{m_{34},m_{56}\}$, $V_5=\{m_{45}\}$, $V_6=\{m_{67}\}$ and $V_7=V$, then $f=(V_1,\dots, V_7)$ is a TDC of $M(P_8)$, and hence, $\chi_d^t(M(P_8))=7$.

Finally, consider the case $n\ge 9$. By the description of total dominating sets from \cite{KPPT}, we have that $\gamma_t(M(P_n))=\lceil \frac{2n}{3}\rceil$ and that if $S$ is a total dominating set of $M(P_n)$, then $M(P_n)[V\setminus S]$ is the disjoint union of graphs isomorphic to $K_1$ and $P_3$. This fact together with Theorem~\ref{chi_d ^t=<gamma+min chi} and Corollary~\ref{chiM_d ^t=>2n/3} implies that 
$$\lceil \frac{2n}{3}\rceil +2\ge \chi_d^t(M(P_n))\ge \lceil \frac{2n}{3}\rceil. $$ 
Consider $f=(V_1,\dots,V_\ell)$ be a minimal TDC of $M(P_n)$, and $S$ a minimal total dominating set of $M(P_n)$. Notice that $|S|=\lceil \frac{2n}{3}\rceil$. By \cite[Lemma 2.1]{KPPT}, we can assume that $S=\{m_{i_1(i_1+1)},\dots, m_{i_{\lceil \frac{2n}{3}\rceil}(i_{\lceil \frac{2n}{3}\rceil}+1)}\}\subseteq M$,  and that each element of $S$ belong to a different color class. Without loss of generalities, assume that $m_{i_j(i_j+1)}\in V_j$.
%By Definition~\ref{total dominator coloring}, if for every $i=1,\dots,\ell$, we pick an element from $V_i$, we construct $S$ a total dominating set of $M(P_n)$.  

Suppose that $\ell=\lceil \frac{2n}{3}\rceil$.
%This implies that $S$ is a minimal total dominating set of $M(P_n)$.  
Since $n\ge9$, then $n-1> \lceil \frac{2n}{3}\rceil$, and the induced subgraph $M(P_n)[V\setminus S]$ has a subgraph $G$ isomorphic to $P_3$ of the form $M(P_n)[v_r,m_{r(r+1)},v_{r+1}]$, for some $r$. Since $G$ needs at least two colors, this implies that $|\{V_i~|~|V_i|\ge2\}|\ge2$. 
Assume that $V_{p}$ and $V_{s}$ have cardinality bigger than $2$ for some $1\le p<s\le \ell$. This implies that $m_{(i_p-1)i_p}$, $m_{(i_p+1)(i_p+2)}$, $m_{(i_s-1)i_s}$ and $m_{(i_s+1)(i_s+2)}$ all belong to color classes made of only one element, otherwise one between $v_{i_p}$, $v_{i_p+1}$, $v_{i_s}$ and $v_{i_s+1}$ is not a common neighbour of any $V_i$. This implies that $m_{(i_p-1)i_p}, m_{(i_p+1)(i_p+2)}, m_{(i_s-1)i_s}, m_{(i_s+1)(i_s+2)}\in S$, and hence that $|S|>\lceil \frac{2n}{3}\rceil$. This implies that $\ell\ge\lceil \frac{2n}{3}\rceil+1$. 

Assume that $\ell=\lceil \frac{2n}{3}\rceil+1$. Let $S'=S\cup\{w\}$, where $w\in V_\ell$. Since $n\ge9$, then $n-1> \lceil \frac{2n}{3}\rceil+1$, and the induced subgraph $M(P_n)[V\setminus S']$ has a subgraph $G$ isomorphic to $P_3$ of the form $M(P_n)[v_r,m_{r(r+1)},v_{r+1}]$, for some $r$. Using the same argument as the case $\ell=\lceil \frac{2n}{3}\rceil$, we have that one between $m_{r(r+1)}$ and $v_r$ belongs to $V_\ell$ and the other to a $V_p$ for some $i\le p\le \ell-1$. In addition, $|V_j|=1$ for all $j\ne \ell, p$, and $m_{(i_p-1)i_p}$ and $m_{(i_p+1)(i_p+2)}$ belong to color classes made of only one element. This implies that $v_{i_p},v_{i_p+1}\in V_\ell$, and hence that one between $m_{(i_p-1)i_p}$ and $m_{(i_p+1)(i_p+2)}$ is not a common neighbour of any $V_i$. This implies that $\ell\ge\lceil \frac{2n}{3}\rceil+2$, and hence that $\chi_d^t(M(P_n))=\lceil \frac{2n}{3}\rceil+2$.
\end{proof}
%---------------------------------------------------------------
%------------------------------------------------------

\begin{Lemma}\label{lemma:TDCcycleandpathcompared} For any $n\ge5$, we have
$$\chi_d^t(M(P_n))\le \chi_d^t(M(C_n))\le n.$$
%$$\chi_d^t(M(P_n))+1\ge \chi_d^t(M(C_n))\ge\chi_d^t(M(P_n)).$$

\end{Lemma}
\begin{proof} To fix the notation, assume $V(P_n)=V(C_n)=V=\{ v_1, \dots, v_n\}$, $E(P_n)=\{v_1v_2, v_2v_3,\dots,v_{n-1}v_n\}$ and $E(C_n)=E(P_n)\cup\{v_1v_n\}$. Then $V(M(P_n))=V\cup \mathcal{M}$, where $\mathcal{M}=\{ m_{i(i+1)}~|~1\leq i \leq n-1 \}$, and $V(M(C_n))=V\cup \mathcal{M}\cup\{m_{1n}\}$.

%If $n=3$, then by Theorem~\ref{chi_d^t(M(P))}, $\chi_d^t(M(P_3))=3$. On the other hand, since the induced subgraph $M(C_3)[v_1,v_2,v_3]$ is isomorphic to $K_3$ and $f=(\{v_1,m_{23}\},\{v_2,m_{13}\},\{v_3,m_{12}\})$ is a TDC of $M(C_3)$, then $\chi_d^t(M(C_3))=3$.

%If $n=4$, then by Theorem~\ref{chi_d^t(M(P))}, $\chi_d^t(M(P_4))=4$. On the other hand, each of the induced subgraphs $M(C_4)[v_1, m_{14},m_{12}]$ and $M(C_4)[v_3, m_{34},m_{23}]$ is isomorphic to $K_3$. This implies that they both need $3$ colors. However, since $N_{M(C_4)}(v_1)=\{m_{14},m_{12}\}$ and $N_{M(C_4)}(v_3)=\{m_{34},m_{23}\}$, then each of the two induced subgraph need a vertex that belong to a color class made of only one element. This implies that $\chi_d^t(M(C_4))\ge4$. On the other hand, $f=(\{m_{12}\},\{m_{23},m_{14}\}, \{m_{34}\},\{v_1,v_2,v_3,v_4\})$ is a TDC of $M(C_4)$, and hence, $\chi_d^t(M(C_4))=4$.

%Assume that $n\ge 6$. Let $f=(V_1,\dots, V_r)$ be a minimal TDC of $M(P_n)$. Then $g=(V_1,\dots, V_r, V_{r+1})$, where $V_{r+1}=\{m_{1n}\}$, is a TDC of $M(C_n)$. This shows that $\chi_d^t(M(P_n))+1\ge \chi_d^t(M(C_n))$. We now want to prove the other inequality.

%Assume that $n\ge 6$ and $f=(V_1,\dots, V_s)$ be a minimal TDC of $M(C_n)$. 
If we define $V_1=\{m_{12}\}, V_2=\{m_{23},m_{1n}\}$, $V_i=\{m_{i(i+1)}\}$ for all $i=3,\dots, n-1$, and $V_n=V$, then $(V_1,\dots, V_n)$ is a TDC of $M(C_n)$. This shows that $\chi_d^t(M(C_n))\le n$. 
Let $f=(V_1,\dots, V_s)$ be a minimal TDC of $M(C_n)$. Since $s\le n$, without loss of generality, we can assume that $m_{1n}\in V_1$ and $|V_1|\ge2$. Define $g=(V'_1,V_2,\dots, V_s)$, where $V'_1=V_1\setminus\{m_{1n}\}$. By construction $g$ is a coloring of $M(P_n)$. This implies that $v_1,v_n\nsucc_t V_1$. Moreover, if $m_{12}\succ_t V_1$, then $m_{12}\succ_t V'_1$, and similarly for $m_{n(n-1)}$. This implies that $g$ is a TDC of $M(P_n)$, and hence $\chi_d^t(M(C_n))\ge\chi_d^t(M(P_n))$.
\end{proof}

%--------------------------------------------------
%----------------------------------------------------

\begin{Theorem} \label{chi_d^t(M(C))} For any cycle $C_n$ of order $n\geq 3$, 
$$
\chi_d^t(M(C_n))=
\begin{cases}
4& \text{ if $ n=3$}\\
n& \text{ if $ n=4, 5$}\\ 
\lceil \frac{2n}{3}\rceil+2 & \text{ otherwise} 
\end{cases}
$$
\end{Theorem}
\begin{proof} To fix the notation, assume $V(C_n)=V=\{ v_1, \dots, v_n\}$ and $E(C_n)=\{v_1v_2, v_2v_3,\dots,v_{n-1}v_n, v_1v_n\}$. Then $V(M(C_n))=V\cup \mathcal{M}$, where $\mathcal{M}=\{ m_{i(i+1)}~|~1\leq i \leq n-1 \}\cup\{m_{1n}\}$.

Assume that $n=3,4$. Then $M(C_n) \supseteq K_3$, so $\chi_d^t(M(C_n))\ge 3$. If $\chi_d^t(M(C_n))= 3$, then $|A_1|=0$ which is a contradiction by Lemma \ref{lemma:A_i}. So $\chi_d^t(M(C_n))\ge 4$. Now since $(\{m_{12}\},\{m_{23}\},\{m_{13}\},V)$ and $(\{m_{12}\},\{m_{23},m_{14}\},\{m_{34}\},V)$ are TDC of $M(C_3)$ and $M(C_4)$ respectively, we have $\chi_d^t(M(C_n))= 4$.

By Lemma~\ref{lemma:TDCcycleandpathcompared}, if $n=5$ then $\chi_d^t(M(C_5))=5$. 

%Assume $n=5$, then by Lemma~\ref{lemma:TDCcycleandpathcompared} and Theorem~\ref{chi_d^t(M(P))}, $\chi_d^t(M(C_5))\ge5$. On the other hand, $f=(\{m_{12}\}, \{m_{23},m_{1,5}\}, \{m_{34}\},\{m_{4,5}\},\{v_1,\dots, v_5\})$ is a TDC of $M(C_5)$, and hence $\chi_d^t(M(C_5))=5$.
%--------------------------------
Assume now that $n\ge 6$ and $n\neq 8$. By the description of total dominating sets from \cite{KPPT}, we have that $\gamma_t(M(C_n))=\lceil \frac{2n}{3}\rceil$ and that we have that if $S$ is a total dominating set of $M(C_n)$, then $M(C_n)[V\setminus S]$ is the disjoint union of graphs isomorphic to $K_1$ and $P_3$. This fact together with Theorem~\ref{chi_d ^t=<gamma+min chi} and Corollary~\ref{chiM_d ^t=>2n/3} implies that 
$$\lceil \frac{2n}{3}\rceil +2\ge \chi_d^t(M(C_n))\ge \lceil \frac{2n}{3}\rceil. $$ 
On the other hand, by Theorem~\ref{chi_d^t(M(P))} and Lemma~\ref{lemma:TDCcycleandpathcompared}
$$\lceil \frac{2n}{3}\rceil +2=\chi_d^t(M(P_n)) \le \chi_d^t(M(C_n)) $$ 
and hence, $\chi_d^t(M(C_n))= \lceil \frac{2n}{3}\rceil+2$, for all $n\ge6$ except for $n=8$.

%Assume now that $n\ge 6$. If $n \equiv 0 \mod 3$, then consider $V_1=\{m_{12}\}$, $V_2=\{m_{23}\}$, $V_3=\{m_{45}\}$, $V_4=\{m_{56}\},\dots, V_{\lceil \frac{2n}{3}\rceil-1}=\{m_{(n-2)(n-1)}\}$, $V_{\lceil \frac{2n}{3}\rceil}=\{m_{(n-1)n}\}$, $V_{\lceil \frac{2n}{3}\rceil+1}=\{m_{ij}~|~m_{ij}\notin V_k\text{ for } 1\le k\le \lceil \frac{2n}{3}\rceil\}$ and $V_{\lceil \frac{2n}{3}\rceil+2}=V$. 
%If $n \equiv 1 \mod 3$, then consider 
%$V_1=\{m_{12}\}$, $V_2=\{m_{23}\}$, $V_3=\{m_{45}\}$, $V_4=\{m_{56}\},\dots, V_{\lceil \frac{2n}{3}\rceil-2}=\{m_{(n-3)(n-2)}\}$, $V_{\lceil \frac{2n}{3}\rceil-1}=\{m_{(n-2)(n-1)}\}$, $V_{\lceil \frac{2n}{3}\rceil}=\{m_{(n-1)n}\}$ $V_{\lceil \frac{2n}{3}\rceil+1}=\{m_{ij}~|~m_{ij}\notin V_k\text{ for } 1\le k\le \lceil \frac{2n}{3}\rceil\}$ and $V_{\lceil \frac{2n}{3}\rceil+2}=V$. 
%If $n \equiv 2 \mod 3$, then consider 
%$V_1=\{m_{12}\}$, $V_2=\{m_{23}\}$, $V_3=\{m_{45}\}$, $V_4=\{m_{56}\},\dots, V_{\lceil \frac{2n}{3}\rceil-1}=\{m_{(n-1)n}\}$, $V_{\lceil \frac{2n}{3}\rceil}=\{m_{1n}\}$, $V_{\lceil \frac{2n}{3}\rceil+1}=\{m_{ij}~|~m_{ij}\notin V_k\text{ for } 1\le k\le \lceil \frac{2n}{3}\rceil\}$ and $V_{\lceil \frac{2n}{3}\rceil+2}=V$. In all three cases, $f=(V_1,\dots, V_{\lceil \frac{2n}{3}\rceil+2})$ is a TDC of $M(C_n)$, and hence $\chi_d^t(M(C_n))\le \lceil \frac{2n}{3}\rceil+2$. By Lemma~\ref{lemma:TDCcycleandpathcompared} and Theorem~\ref{chi_d^t(M(P))}, we can conclude that $\chi_d^t(M(C_n))= \lceil \frac{2n}{3}\rceil+2$, for all $n\ge6$ except for $n=8$.

Finally for $n=8$, by the previous argument, $8\ge\chi_d^t(M(C_8))\ge 7$. Let $f=(V_1,\dots, V_\ell)$ be a minimal TDC of $M(C_n)$. If $\ell= 7$, then there exists $1\le i\le 7$ such that $|V_i\cap M|\ge2$. Without loss of generalities, assume $i=1$ and $m_{12}\in V_1$. Since $f$ is a TDC, then we can assume $V_2=\{m_{18}\}$ and $V_3=\{m_{23}\}$. To finish the proof it is enough to discuss the case when $m_{67}\in V_1$ and when $m_{78}\in V_1$. 
Suppose that $m_{67}\in V_1$. This implies that we can assume $V_4=\{m_{78}\}$ and $V_5=\{m_{56}\}$. As a consequence $V_6=\{m_{34}\}$ or $V_6=\{m_{45}\}$. If $V_6=\{m_{34}\}$, then $V_7=\{m_{45}\}$ or $V_7=\{v_5\}$. However in both cases, we do not have enough color classes to color all remaining vertices. If $V_6=\{m_{45}\}$, then $V_7=\{m_{34}\}$ or $V_7=\{v_3\}$. However in both cases, we do not have enough color classes to color all remaining vertices.
Suppose now that $m_{78}\in V_1$. This implies that we can assume $V_4=\{m_{67}\}$. As a consequence $V_5=\{v_2\}$ or $V_5=\{v_3\}$ or $V_5=\{m_{34}\}$, and similarly, $V_6=\{m_{45}\}$ or $V_6=\{m_{56}\}$. All these cases imply that $v_1,v_8\in V_7$ and that $|V_1|,|V_7|\ge3$. However, this implies that $m_{18}$ is not a common neighbour of any $V_i$.
This shows that $\chi_d^t(M(C_8))=8$
\end{proof}
%--------------------------------------------------------
%-------------------------------------------------------
\begin{Remark} The two inequalities of Lemma~\ref{lemma:TDCcycleandpathcompared} are both strict, by Theorem \ref{chi_d^t(M(C))}. %In fact, $\chi_d^t(M(P_3))=3=\chi_d^t(M(C_3))$, and $\chi_d^t(M(P_8))=7$ but $\chi_d^t(M(C_8))=8$.
\end{Remark}
%---------------------------------------------------------------------------
\begin{Theorem} \label{chi_d^t(M(W))} For any wheel $W_n$ of order $n\geq 4$, 
$$
\chi_d^t(M(W_n))=
\begin{cases}
5 & \text{ if $n=4$}\\ 
n+2 & \text{ if $n\ge 5$}
\end{cases}
$$
\end{Theorem}
\begin{proof}
To fix the notation, assume $V(W_n)=V=\{v_0,v_1,\dots, v_{n-1}\}$ and $E(W_n)=\{v_0v_1,v_0v_2,\dots, v_0v_{n-1}\}\cup\{v_1v_2, v_2v_3,\dots,v_{n-1}v_1\}$. Then we have $V(M(W_n))=V(W_n)\cup \mathcal{M}$, where $\mathcal{M}=\{ m_i~|~1\leq i \leq n-1 \}\cup\{ m_{i(i+1)}~|~1\leq i \leq n-2 \}\cup\{m_{1(n-1)}\}$. 

Assume $n=4$. Consider $V_1=V$, $V_2=\{m_1,m_{23}\}$, $V_3=\{m_2\}$, $V_4=\{m_3,m_{12}\}$ and $V_5=\{m_{13}\}$. By construction, $f=(V_1,\dots,V_{5})$ is a TDC of $M(W_n)$, and hence $\chi_d^t(M(W_4))\le 5$. On the other hand, let $f=(V_1,\dots,V_{\ell})$ be a minimal TDC of $M(W_4)$. 
%By Lemma \ref{3item} $|A_1| \geq2$ and so $\ell \geq 5$. This implies that $\chi_d^t(M(W_4))= 5$.} 
Since $K_{4}\cong M(W_4)[v_0,m_1,\dots, m_{3}] \subseteq M(W_4)$, then $\ell\geq 4$. If $\ell=4$, this implies that, up to reordering the color classes, $V_1=V$, $V_2=\{m_1,m_{23}\}$, $V_3=\{m_2, m_{13}\}$ and $V_4=\{m_3,m_{12}\}$, but in this way $v_0\nsucc_t V_i$ for $1\leq i \leq 4$ and so $f$  is not a TDC of $M(W_4)$ which is a contradiction. This implies that $\chi_d^t(M(W_4))= 5$.

Assume that $n\ge 5$ is even. Consider $V_1=V$, $V_2=\{m_1,m_{23},m_{45},\dots,$ $m_{(n-2)(n-1)}\}$, $V_3=\{m_{34},m_{56},\dots, m_{(n-1)n}\}$, $V_4=\{m_{12}\}$ and, for any $i=5,\dots, n+2$, $V_i=\{m_{i-3}\}$. By construction, $f=(V_1,\dots,V_{n+2})$ is a TDC of $M(W_n)$, and hence $\chi_d^t(M(W_n))\le n+2$.
 On the other hand, let $f=(V_1,\dots,V_{\ell})$ be a minimal TDC of $M(W_n)$. Since $K_{n}\cong M(W_n)[v_0,m_1,\dots, m_{n-1}] \subseteq M(W_n)$, then $\ell\geq n$ and $v_0,m_1,\dots, m_{n-1}$ all belong to different color classes. If each $m_i$ belongs to a color class of cardinality $1$, then $\ell\ge n+3$, contradicting the first part of the proof for this case. This implies that there exists at least one $m_i$ belonging to a color class with at least $2$ elements. 
 If there is only one such class, then we can assume, without loss of generality, that $m_1\in V_1$ and $|V_1|\ge 2$. Since $f$ is a TDC, then at least one between $m_{12}$ and $m_{1(n-1)}$ has to belong to a color class of cardinality $1$. Moreover, $m_{12}, m_{1(n-1)}\notin V_1$ and hence $\ell\ge n+2$, proving our thesis. 
 Assume that there exist $1\le i\le n-1$ such that $m_i, m_j$ belong to a color class with at least $2$ elements. If $j-i=1$ or $j-i=n-2$, then we can assume, without loss of generality, that $i=1$, $j=2$, $m_1\in V_1$ with $|V_1|\ge 2$ and $m_2\in V_2$ with $|V_2|\ge 2$. This implies that two vertices between $v_1, v_2, m_{1(n-1)}, m_{12}$ and $m_{23}$ have to belong to a color class of cardinality $1$ and hence $\ell\ge n+2$, proving our thesis. 
 If $2\ge j-i\ge n-3$, then at least one between $m_{(i-1)i}$ and $m_{i(i+1)}$ has to belong to a color class of cardinality $1$, and similarly at least one between $m_{(j-1)j}$ and $m_{j(j+1)}$ has to belong to a color class of cardinality $1$. This implies that $\ell\ge n+2$, proving our thesis. This shows that $\chi_d^t(M(W_n))= n+2$.

Assume $n$ is odd. Consider $V_1=V$, $V_2=\{m_{12},m_{34},\dots, m_{(n-2)(n-1)}\}$, $V_3=\{m_{23},m_{45},\dots, m_{(n-1)n}\}$ and, for any $i=4,\dots, n+2$, $V_i=\{m_{i-3}\}$. By construction, $f=(V_1,\dots,V_{n+2})$ is a TDC of $M(W_n)$, and hence $\chi_d^t(M(W_n))\le n+2$. 
On the other hand, let $f=(V_1,\dots,V_{\ell})$ be a minimal TDC of $M(W_n)$. %\textcolor{blue}{Similar to proof case $n$ is even, we have $\ell\ge n+2$. This shows that $\chi_d^t(M(W_n))= n+2$.(The proof is quite similar. I think we should delete this part.) }. 
Since $K_{n}\cong M(W_n)[v_0,m_1,\dots, m_{n-1}] \subseteq M(W_n)$, then $\ell\geq n$ and $v_0,m_1,\dots, m_{n-1}$ all belong to different color classes. If each $m_i$ belongs to a color class of cardinality $1$, then $\ell\ge n+2$, proving our thesis. Assume that there exists at least one $m_i$ belonging to a color class of cardinality bigger than $2$.
 If there is only one such class, then we can assume, without loss of generality, that $m_1\in V_1$ and $|V_1|\ge 2$. Since $f$ is a TDC, then at least one between $m_{12}$ and $m_{1(n-1)}$ has to belong to a color class of cardinality $1$. Moreover, $m_{12}, m_{1(n-1)}\notin V_1$ and hence $\ell\ge n+2$, proving our thesis. Assume that there exist $1\le i\le n-1$ such that $m_i, m_j$ belong to a color class of cardinality bigger than $2$. If $j-i=1$ or $j-1=n-2$, then we can assume, without loss of generality, that $i=1$, $j=2$, $m_1\in V_1$ with $|V_1|\ge 2$ and $m_2\in V_2$ with $|V_2|\ge 2$. 
 This implies that two vertices between $v_1, v_2, m_{1(n-1)}, m_{12}$ and $m_{23}$ have to belong to a color class of cardinality $1$ and hence $\ell\ge n+2$, proving our thesis. If $2\ge j-i\ge n-3$, then at least one between $m_{(i-1)i}$ and $m_{i(i+1)}$ has to belong to a color class of cardinality $1$, and similarly at least one between $m_{(j-1)j}$ and $m_{j(j+1)}$ has to belong to a color class of cardinality $1$. This implies that $\ell\ge n+2$, proving our thesis. This shows that $\chi_d^t(M(W_n))= n+2$.
\end{proof}
%---------------------------------------------------
%------------------------------------------------

\begin{Theorem}[\cite{Clark}]\label{chi'  (K_n)=n,n-1}
	For any complete graph  $K_{n}$ of order at least 2,
	\begin{equation*}
	\chi'  (K_n)=\left\{
	\begin{array}{ll}
	n-1 & \mbox{if $n$ is even}\\
	n  & \mbox{if $n$ is odd}.
	\end{array}
	\right.
	\end{equation*}
\end{Theorem}

\begin{Theorem}\label{chi(M(K_{n}))=n}  For any complete graph $K_{n}$ on $n$ vertices, with $n\geq 2$, 
	\[
	\chi(M(K_{n}))=n.
	\]
\end{Theorem}

\begin{proof}
To fix the notation, assume $V(K_n)=V=\{v_1,\dots, v_{n}\}$ and $E(K_n)=\{v_iv_j| 1 \leq i < j \leq n\}$. Then we have $V(M(K_n))=V(K_n)\cup \mathcal{M}$, where $\mathcal{M}=\{ m_{ij}~|~1\leq i<j \leq n \}$. 

Assume $n$ is even. By Theorem~\ref{chi'  (K_n)=n,n-1}, we can consider $f_1=(V_1,\dots,V_{n-1})$ a proper coloring of $L(K_n)$. Since to color the induced subgraph $M(K_n)[v_1,m_{12},\dots, m_{1n}] \subseteq M(K_n)$ we need $n$ colors, then $\chi(M(K_{n})) \geq n$. On the other hand, $g_1=(V_1,\dots,V_{n-1},V)$  is a proper coloring of $M(K_n)$, and hence $\chi(M(K_{n}))=n$.

Assume $n$ is odd. By Theorem~\ref{chi'  (K_n)=n,n-1}, we can consider $f_2=(V_1,\dots,V_{n})$ a proper coloring of $L(K_n)$. Since to color the induced subgraph $M(K_n)[v_1,m_{12},\dots, m_{1n}] \subseteq M(K_n)$ we need $n$ colors, then $\chi(M(K_{n})) \geq n$. On the other hand, for each $1\le i\le n$, $d_{K_n}(v_i)=n-1$. This implies that for each $1\le i\le n$ there exists $1\le j\le n$ such that $m_{ik}\notin V_j$ for all $1\le k\le n$. Let $W_j=V_j\cup\{v_i~|~m_{ik}\notin V_j \text{ for all }  1\le k\le n\}$. By construction $g_2=(W_1,\dots,W_n)$ is a proper coloring of $M(K_n)$, and hence $\chi(M(K_{n}))=n$.
\end{proof}
%----------------------------------------------------
%-------------------------------------------
%By \cite{KPPT} we know that $\gamma_t(M(K_n))= \lceil \frac{2n}{3}\rceil$, so by Theorems \ref{2=<chi_d ^t=<n}, \ref{chi_d ^t=<gamma+min chi} and \ref{chi(M(K_{n}))=n} we have the following result.
%%------------------------------------------------------------------
%
%\begin{Corollary}\label{boundforMKn} For any $n\ge2$, we have
%	$$n \le \chi_d^t(M(K_n))\le n+\lceil \frac{2n}{3}\rceil.$$
%\end{Corollary}

%--------------------------------------------------------------
%\begin{Lemma}\label{alpha(LKn)} For any complete graph $K_n$ of order $n \ge 2$, $$\alpha(L(K_n))=\lfloor \frac{n}{2}\rfloor.$$
%\end{Lemma}
%--------------------------------------
%\begin{proof}
%Assume $V(K_n)=V=\{v_1,\dots, v_{n}\}$ and $E(K_n)=\{v_iv_j| 1 \leq i < j \leq n\}$. Then $V(L(K_n))=\mathcal{M}=\{ m_{ij}~|~1\leq i<j \leq n \}.$ Let $S$ be a independent set of $L(K_n)$. Since $\{p,q\} \cap \{r,k\}=\emptyset$ for every $m_{pq},m_{rk} \in S $, we conclude $|S| \le \lfloor \frac{n}{2}\rfloor$. On the other hand, since the set $\{m_{(2i-1)(2i)}| 1 \le i \le \lfloor \frac{n}{2}\rfloor\}$ is an independent set with cardinality $\lfloor \frac{n}{2}\rfloor$, we obtain $\alpha(L(K_n))=\lfloor \frac{n}{2}\rfloor.$
%\end{proof}
%---------------------------------------------------
\begin{Theorem}\label{new bound for Kn}
For any complete graph $K_n$ of order $n \ge 2$, 
$$n+1\le\chi_d^t(M(K_n))\le n+\lceil \frac{2n}{3}\rceil-1.$$
Moreover, the bounds are tight.
\end{Theorem}
\begin{proof}
To fix the notation, assume $V(K_n)=V=\{v_1,\dots, v_{n}\}$ and $E(K_n)=\{v_iv_j| 1 \leq i < j \leq n\}$. Then $V(M(K_n))=V \cup \mathcal{M}$ where $\mathcal{M}=\{ m_{ij}~|~1\leq i<j \leq n \}.$ By \cite{KPPT}, we know that $\gamma_t(M(K_n))= \lceil \frac{2n}{3}\rceil$ and also that the sets
\begin{equation*}
\begin{array}{ll}
S_0=\{m_{(3i+1)(3i+2)}, m_{(3i+2)(3i+3)}~|~ 0 \leq i \leq \lfloor n/3 \rfloor-1\}& \mbox{if }n \equiv 0 \pmod{3}, \\
S_1=S_0\cup\{m_{(n-1)n}\} & \mbox{if }n \equiv 1 \pmod{3}, \\
S_2=S_0\cup\{m_{(n-2)(n-1)},m_{(n-1)n}\} & \mbox{if }n \equiv 2 \pmod{3},
\end{array}
\end{equation*}
are minimal TDSs of $M(K_n)$. Since to color the induced subgraph $(M(K_n)-S_i)[v_1,m_{13},\dots, m_{1n}] \subseteq M(K_n)-S_i$ for each $0 \le i \le 2$, we need $n-1$ colors, then $\chi(M(K_{n})-S_i) \geq n-1$. %In continue, we show that $\chi(M(K_{n})-S_i) \le n-1$. 
We know that for every graph $G$, $\chi(G) \le \Delta(G)+1$. Now since $\Delta(M(K_{n})-S_i)=n-2$, we have $\chi(M(K_{n})-S_i) \le n-1$, and hence $\chi(M(K_{n})-S_i) = n-1$.
% Let $h=(V_1,\dots,V_{\ell})$ be a proper coloring of $L(K_n)-S_i$ for each $0 \le i \le 2$. Then $|V_j| \le \lfloor \frac{n}{2}\rfloor$ for each $1 \le j \le \ell$. Also since for any independent set $S^{\prime}$ of $L(K_n)-S_i$, similar to the proof Lemma \ref{alpha(LKn)} we can prove $|S^{\prime}| \le \lfloor \frac{n}{2}\rfloor$, and $\{m_{(i)(\lfloor \frac{n}{2}\rfloor+i)}|1 \leq i \leq \lfloor \frac{n}{2}\rfloor \}$ is an independent set of $L(K_n)-S_i$ with cardinality $\lfloor \frac{n}{2}\rfloor$, we obtain $\alpha(L(K_n)-S_i)=\lfloor \frac{n}{2}\rfloor$. By considering $$|V(L(K_n)-S_i)|=\dfrac{(n)(n-1)}{2}-\lceil \frac{2n}{3}\rceil  $$ we have $|V(L(K_n)-S_i)| \le (n-2) \lfloor \frac{n}{2}\rfloor $. On the other hand, since to color the induced subgraph $L(K_n)-S_i[m_{13},\dots, m_{1n}] \subseteq L(K_n)-S_i$ for each $0 \le i \le 2$, we need $n-2$ colors, so $\chi(L(K_{n})-S_i) = n-2$. Now let $f_1=(V_1,\dots,V_{n-2})$ be a proper coloring of $L(K_n)-S_i$ for each $0 \le i \le 2$. Since $f_2=(V_1,\dots,V_{n-2},V)$ is a proper coloring of $M(K_n)-S_i$ for each $0 \le i \le 2$, we have $\chi(M(K_{n})-S_i) = n-1$.
 By Theorem \ref{chi_d ^t=<gamma+min chi} we have $\chi_d^t(M(K_n))\le n+\lceil \frac{2n}{3}\rceil-1$. 
 %Moreover, by Theorems \ref{2=<chi_d ^t=<n} and \ref{chi(M(K_{n}))=n} we have $n\le\chi_d^t(M(K_n))$.

We claim that $n+1\le\chi_d^t(M(K_n))$. By absurd, assume $\chi_d^t(M(K_n))=n=\chi(M(K_n))$ and let $f=(V_1,\dots,V_{n})$ be a minimal TDC of $M(K_n)$. By Theorems \ref{chi'  (K_n)=n,n-1} and \ref{chi(M(K_{n}))=n}, this implies that $|V_i| \ge 2$ for each $1 \le i \le n$, which is a contradiction by Lemma \ref{lemma:A_i}. As a consequence, $n+1\le\chi_d^t(M(K_n))$. By Theorem \ref{chi_d^t(M(C))}, the bounds are tight when $n=3$.
\end{proof}
%--------------------------------------------------------------------
%-------------------------------------------------------
\begin{Definition} The \emph{friendship} graph $F_n$ of order $2n+1$ is obtained by joining $n$ copies of the cycle graph $C_3$ with a common vertex.
\end{Definition}
%--------------------------------

\begin{Theorem}\label{prop:mintotdominatorfriendship}
Let $F_n$ be the friendship graph with $n\ge2$. Then $$\chi_d^t(M(F_n))= 2n+2.$$
\end{Theorem}
\begin{proof} To fix the notation, assume $V(F_n)=\{v_0,v_1,\dots, v_{2n}\}$ and $E(F_n)=\{v_0v_1,v_0v_2,\dots, v_0v_{2n}\}\cup\{v_1v_2, v_3v_4,\dots,v_{2n-1}v_{2n}\}$. Then $V(M(F_n))=V(F_n)\cup \mathcal{M}$, where $\mathcal{M}=\{ m_i~|~1\leq i \leq 2n \}\cup\{ m_{i(i+1)}~|~1\leq i \leq 2n-1 \text{ and } i \text{ is odd}\}$. 

Let $f=(V_1,\dots,V_{\ell})$ be a minimal TDC of $M(F_n)$. Since the induced subgraph $M(F_n)[v_0,m_{1},\dots, m_{2n}]$ is isomorphic to $K_{2n+1}$, we have $\ell \ge 2n+1$. Suppose $\ell = 2n+1$. Without loss of generality, we can assume that $m_i \in V_i$ for $1 \le i \le 2n$ and $v_0 \in V_{2n+1}$. This implies that $v_i \succ_t \{m_i\}$ for $1 \le i \le 2n$ and $\{m_i\}$ is an unique color class. Now since at least two color classes are needed to color the vertices $v_i$ and $m_{i(i+1)}$ for some $i$, we have $\ell > 2n+1$, which is a contradiction. So $\ell \ge 2n+2$. 
On the other hand, if we consider $V_i=\{m_i\}~~~\text{for}~~1 \le i \le 2n$, $V_{2n+1}=\{v_0 , v_1 , \dots, v_{2n}\}$, and 
$V_{2n+2}=\{ m_{i(i+1)}~|~1\leq i \leq 2n-1 \text{ and } i \text{ is odd}\},$
we have that $(V_1,\dots,V_{2n+2})$ is a TDC of $M(F_n)$, and hence $\chi_d^t(M(F_n))= 2n+2.$
\end{proof}

\section{Middle graph of trees}

In this section, we will describe bounds for the total dominator coloring number of the middle graph of trees.

\begin{Theorem}  \label{chi_d^t(M(T))< n}Let $T$ be a tree of order $n\geq 2$. Then
$$\chi_d^t(M(T))\leq n.$$
\end{Theorem}
\begin{proof}
Assume $V(T)=\{v_1,\dots,v_n\}$. Then $V(M(T))=V(T)\cup M$ where
$M=\{m_{ij}~|~v_iv_j\in E(T)\}$. We give $|M|=n-1$ colors to each element of $M$ and a different color to each element in $V(T)$. This coloring is a TDC of $M(T)$ with $n$ color classes. Hence $\chi_d^t(M(T))\leq n$.

\end{proof}

 If we consider $T$ a tree and we denote the set of leaves of $T$ by $\leaf(T)=\{v\in V(T)~|~d_T(v)=1\}$, then we have the following result.
\begin{Theorem}\label{prop:mintotdominatortreeleaf}
Let $T$ be a tree with $n\ge3$ vertices. Then $$\chi_d^t(M(T))\ge |\leaf(T)|+1.$$
\end{Theorem}
\begin{proof} To fix the notation, assume $\leaf(T)=\{v_1,\dots,v_k\}$, for some $k\le n$ and $f=(V_1,\dots,V_{\ell})$ be a minimal TDC of $M(T)$. Since for each $i=1,\dots, k$, $N_{M(T)}(v_i)=\{m_{ij}\}$ for some $j$, we have that $\{m_{ij}\}$ is an unique color class so that $v_i \succ_t \{m_{ij}\}$. This implies that $\ell \ge k$. On the other hand, since at least one color class is needed to color the vertices in $V(M(T))-\{N_{M(T)}(v_i)| 1 \le i \le k\}$, we have $\chi_d^t(M(T))\ge k+1= |\leaf(T)|+1$.
\end{proof}

\begin{Remark}
Notice that the inequality of Theorem \ref{prop:mintotdominatortreeleaf} is sharp by Theorem \ref{theo:TDCstargraph}.
\end{Remark}

\begin{Theorem} \label{chi_d^t(M(T))=n} For any non-empty tree $T$ of order $n\geq 2$ with $\diam(T)\leq 3$, 
$$\chi_d^t(M(T))=n.$$
\end{Theorem}
\begin{proof}
Assume $V=V(T)=\{v_1,\dots,v_n\}$. Then $V(M(T))=V\cup M$ where
$M=\{m_{i,j}|v_iv_j\in E(T)\}$. 

If $\diam(T)= 1$, then $T\cong K_2$ and $M(T)\cong P_3$, and so $\chi_d^t(M(T))=2$. 

If $\diam(T)= 2$, then $n\geq 3$ and $T\cong K_{1,n-1}$ and so $\chi_d^t(M(T))=n$, by Theorem~\ref{theo:TDCstargraph}. 

If $\diam(T)= 3$, then $T$ is a tree which is obtained by joining central vertex $v$ of a tree $K_{1,p}$ and the central vertex $w$ of a tree $K_{1,q}$ where $p+q=n-2$. Let $\leaf(T)=\{v_i~|~ 1\leq i\leq n-2\}$ be the set of leaves of $T$. Obviously $V(T)=\leaf(T)\cup \{v,w\}$
and $|\leaf(T)|=n-2$. Define $v_{n-1}=v$ and $v_n=w$. Let  $f=(V_1,\dots,V_{\ell})$ be a TDC of $M(T)$. Since, for each $1\le i\le n-2$, if $v_{i}\succ_t V_{k_i}$ for some $1\le k_i\le \ell$, then $|V_{k_i}|=1$ and $V_{k_i}$ is an unique color class for each $1\leq i\leq n-2$, then $\chi_d^t(M(T))\geq n-2$. Since the induced subgraph of $M(T)$ by $\{m_{{(n-1)}n},v_n\}$ is a complete graph of order $2$, we conclude that we need another $2$ colors, so $\chi_d^t(M(T))\geq n$. On the other hand, since $g=(V_1,\dots,V_n)$, where $V_i=\{m_{ij}\}$ for $1\leq i\leq n-2$, $n\leq j\leq n-1$,  $V_{n-1}=\{m_{{(n-1)}n}\}$ and $V_n=V$,  is a TDC of $M(T)$, we have $\chi_d^t(M(T))\leq n$, and hence $\chi_d^t(M(T))=n$.
\end{proof}
 %------------------------------------------------------
\begin{Remark}
Notice that the inequality of Theorem \ref{chi_d^t(M(T))< n} is sharp by Theorem \ref{chi_d^t(M(T))=n}.
\end{Remark}
%---------------------------------
\begin{Remark}
In general, the converse implication of Theorem \ref{chi_d^t(M(T))=n} is not true. To see this it is enough to look at Theorem \ref{chi_d^t(M(S_{1,n,n}))=2n+1}. In fact if $T=S_{1,n,n}$ with $n \ge 2$, then $\chi_d^t(M(T))=|V(T)|$, even if $\diam(T)= 4$. 
\end{Remark}
%--------------------------------------------------------------------

%------------------------------------------------------------------------------
%\section{Nordhaus-Gaddum-like relations}
%
%Finding a Nordhaus-Gaddum like relation for any parameter in graph theory is one of a tradition work which is started after the following theorem by Nordhaus and Gaddum in 1956 \cite{Nordhaus}.
%
%The \emph{complement} of a graph $G$, denoted by $\overline{G}$, is a graph with the vertex set $V(G)$ and for every two vertices $v$ and $w$, $vw\in E(\overline{G})$\ if and only if $vw\not\in E(G)$. 

%-----------------------------------
%-------------------------------------------------------------------------
%\section{Further research}
%
% We finish our discussion with some problems for further research.

%---------------------------------

%-------------------------------------------------------------------------
%-------------------------------------------------------------------------
%\section{Acknowledgments.}
% The authors would like to thank the referee for his/her
%careful reading and valuable comments which improved the quality of the manuscript.
%%-----------------------------------------------------------
\bigskip
\paragraph{\textbf{Acknowledgements}} During the preparation of this article the last author was supported by JSPS Grant-in-Aid for Early-Career Scientists (19K14493).

%%%%%%%%%%%% REFERENCES %%%%%%%%%%%%%%%%%%%%%5

\end{document}